\documentclass[11pt,a4paper]{article}
\usepackage[T1]{fontenc}
\usepackage[utf8]{inputenc}
\usepackage{lmodern,microtype}
\usepackage[margin=28mm]{geometry}
\usepackage{amsmath,amssymb,amsthm,mathtools,mathrsfs,booktabs,array,longtable}
\usepackage{xcolor}
\usepackage[colorlinks=true,linkcolor=blue!45!black,citecolor=blue!45!black,urlcolor=blue!45!black]{hyperref}
\usepackage{enumitem}
\setlist{itemsep=3pt,topsep=5pt}
\allowdisplaybreaks[2]
\newtheorem{theorem}{Theorem}[section]
\newtheorem{proposition}[theorem]{Proposition}
\newtheorem{lemma}[theorem]{Lemma}
\newtheorem{corollary}[theorem]{Corollary}
\theoremstyle{definition}
\newtheorem{definition}[theorem]{Definition}
\newtheorem{remark}[theorem]{Remark}
\newcommand{\R}{\mathbb R}
\newcommand{\C}{\mathbb C}
\newcommand{\Z}{\mathbb Z}

\newcommand{\Pin}{\operatorname{Pin}}
\newcommand{\Dom}{\operatorname{Dom}}
\newcommand{\ran}{\operatorname{ran}}
\newcommand{\ind}{\operatorname{ind}}

\newcommand{\spec}{\operatorname{spec}}

\newcommand{\sx}{\sigma_1}
\newcommand{\sy}{\sigma_2}
\newcommand{\sz}{\sigma_3}
\newcommand{\norm}[1]{\left\lVert#1\right\rVert}
\newcommand{\ip}[2]{\left\langle#1,#2\right\rangle}
\title{Pin Equivariance and Local Dirac Boundary Problems\\on the Flat M\"obius Band}
\author{Anik Chakraborty\\\small Department of Mathematics, University of Delhi, India}
\date{September 16, 2026}
\begin{document}
\maketitle
\begin{abstract}
We study complex rank-two pinor fields on the flat M\"obius band and the relation between their deck equivariance and Dirac boundary realizations. With explicit Clifford conventions, the two lifts for each Pin type are represented by $R=c\sigma_1$, where $c=\pm1$ for $\Pin^+$ and $c=\pm i$ for $\Pin^-$. The relevant longitudinal monodromy is $cJ$, with $J\phi(w)=\sigma_1\phi(-w)$, rather than the fiber matrix alone. Its eigenspaces give shifts $0,1/2$ in the positive case and $1/4,3/4$ in the negative case. We identify the resulting geometric Sobolev spaces and construct self-adjoint elliptic realizations of $D=-i(\sigma_1\partial_x+\sigma_2\partial_w)$ with a local boundary condition defined by the longitudinal tangent field. An exact graph estimate gives compact resolvent and Fredholm mappings on nonnegative Sobolev levels. The spectrum consists of a single unpaired transverse branch and paired branches whose longitudinal shifts alternate with transverse parity. Consequently, the kernel and spectral symmetry depend on the lift and boundary condition; the two negative Pin lifts have eta values $1/2$ and $-1/2$ in the stated convention. We also give a secular equation and an exact kernel criterion for the full translation-invariant family of local self-adjoint elliptic boundary lines.
\end{abstract}

\section{Introduction}
An orientation-reversing identification does not by itself impose antiperiodicity on a pinor. The lifted identification acts both on the fiber and on the transverse variable. The distinction matters when one passes from a geometric quotient to Fourier coefficients, and it matters again when a differential expression is assigned a boundary domain.

We consider
\begin{equation}\label{eq:model}
 \widetilde M=\R\times[-1,1],\qquad
 \tau(x,w)=(x+1,-w),\qquad M=\widetilde M/\langle\tau\rangle,
\end{equation}
with the metric induced by $dx^2+dw^2$. The core has length one and the single boundary component has length two. Our question is precise: for each Pin lift, which Fourier sectors occur, and how do they combine with a specified local elliptic boundary condition to determine a closed Dirac realization?

The answer involves a transverse involution. In the representation fixed below the lift matrix is $R=c\sx$, and
\[
 J\phi(w)=\sx\phi(-w),\qquad J^2=I.
\]
On the $j$-eigenspace of $J$, $j=\pm1$, translation by one acts by $cj$. Thus the shift $\theta_{c,j}\in[0,1)$ is determined by
\begin{equation}\label{eq:phaseintro}
 e^{2\pi i\theta_{c,j}}=cj.
\end{equation}
This gives integer and half-integer sectors for $\Pin^+$, and quarter- and three-quarter-shifted sectors for $\Pin^-$. The two characters within a Pin type interchange the assignment of shifts to the transverse eigenspaces.

We use the local condition $\sx\Psi=b\Psi$ on the physical boundary, where $b=\pm1$. It is globally defined because the longitudinal unit vector field descends to $M$. This choice permits direct proofs of self-adjointness and compact resolvent, and an explicit spectrum. For transverse level $n\geq0$, the shift satisfies
\[
 e^{2\pi i\theta_n}=cb(-1)^n.
\]
For $n\geq1$ there are paired eigenvalues
\[
 \pm\sqrt{\{2\pi(k+\theta_n)\}^2+(n\pi/2)^2},\qquad k\in\Z.
\]
At $n=0$ the eigenvalues are $2\pi b(k+\theta_0)$, with no second transverse state. This branch accounts for both the possible kernel and the spectral asymmetry. These formulas are derived after the domain is established; they are not formulas for APS boundary conditions.

\paragraph{Literature and scope.}
The obstruction and torsor description of Pin structures is classical; see Kirby and Taylor \cite{KT}. Equivariant descent from the universal cover is likewise part of established Pin geometry. Cahen, Gutt and Trautman \cite[Theorem 3]{CGT} give a general covering-space criterion and explain the modified Dirac construction associated with the twisted adjoint action. Our representation conventions make that covariance explicit in dimension two.

Krau\ss har \cite{Krausshar} studies Clifford-analytic constructions on non-orientable quotients, including M\"obius strips. His treatment of the M\"obius case distinguishes harmonic sections from the monogenic constructions available for other quotients. It is not a ready-made spectral theorem for the rank-two boundary realization considered here.

In graphene models, de Souza and Furtado \cite{SouzaFurtado} already obtain periodic and antiperiodic contributions coupled to transverse quantum numbers. Thus neither their coexistence nor a dependence on transverse parity is a new general phenomenon. Their valley and edge conditions differ from the Pin bundle and boundary line fixed here. Araki, Fukaya, Onogi and Yamaguchi \cite{Araki} also treat negative Pin structures and M\"obius geometries, using lattice and domain-wall operators in their study of the Arf--Brown--Kervaire invariant.

The general analytic framework for elliptic Dirac boundary problems is developed by B\"ar and Ballmann \cite{BB,BBGuide}. Local self-adjoint boundary conditions and regularity have also been analyzed systematically by Gro\ss e, Uribe and van den Bosch \cite{GUV}. The existence of such realizations is therefore not a new general theorem. Here the principal results are an explicit joint description of the lift, reflection sector, geometric Sobolev space, and local boundary spectrum, together with a kernel criterion for the translation-invariant boundary family. The results are model-specific; no priority claim for the general mechanisms is made.

Earlier geometric motivation \cite{Chakraborty} and an abstract half-integer Hilbert-scale construction \cite{CK} provide two related settings. Neither is needed for the proofs. Section~\ref{sec:scale} identifies the precise scalar restriction that leads from the present geometry to a prescribed shifted scale.

\section{M\"obius geometry and Pin structures}
\subsection{Topology and characteristic classes}
The action in \eqref{eq:model} is free: $\tau^n(x,w)=(x+n,(-1)^nw)$ has a fixed point only for $n=0$. It is properly discontinuous because a compact subset of the strip meets only finitely many of its integer translates. Hence the quotient is a smooth compact manifold with boundary. Its usual rectangle presentation is the M\"obius band. The strip is contractible, so it is the universal cover and $\pi_1(M)\cong\Z$.

The homotopy $(x,w)\mapsto(x,(1-t)w)$ commutes with $\tau$ and descends to a deformation retraction onto the core $\R/\Z$. Consequently
\[
 H^1(M;\Z_2)=\Z_2,\qquad H^2(M;\Z_2)=0.
\]
The derivative of the generator is
\begin{equation}\label{eq:A}
 A=d\tau=\begin{pmatrix}1&0\\0&-1\end{pmatrix}.
\end{equation}
The orientation character sends $n$ to $(-1)^n$. If $a$ denotes the nonzero class in $H^1(M;\Z_2)$, then
\begin{equation}\label{eq:classes}
 w_1(TM)=a,\qquad w_2(TM)=0,\qquad a^2=0.
\end{equation}
One can also read this from $TM=\underline\R\,\partial_x\oplus L$, where $L$ has monodromy $-1$ around the core. The two boundary lines of the cover are exchanged by $\tau$ and give one boundary circle downstairs. The orientation double cover is $(\R/2\Z)\times[-1,1]$, and is distinct from the universal cover.

\subsection{Principal lifts}
Let $\mathrm{Cl}_{\delta}(2)$ be the real Clifford algebra with
\[
 uv+vu=2\delta\ip{u}{v},\qquad \delta\in\{+1,-1\}.
\]
Its Pin group maps to $O(2)$ by the twisted adjoint action
\[
 p_\delta(g)v=\alpha(g)vg^{-1},
\]
where $\alpha$ is the grading automorphism. In particular, a unit vector $u$ covers reflection in the line perpendicular to $u$, and $u^2=\delta$.

The standard obstruction criteria are $w_2=0$ for $\Pin^+$ and $w_2+w_1^2=0$ for $\Pin^-$. When the obstruction vanishes, the isomorphism classes of lifts of a fixed orthogonal frame bundle form an $H^1(M;\Z_2)$-torsor \cite[Section 1]{KT}. Thus both types exist here and each has two classes. The following construction identifies them directly.

\begin{theorem}[Lifted deck actions]\label{thm:lifts}
Fix $\delta\in\{+1,-1\}$. In the coordinate frame on the cover, the two Pin structures are
\[
 P_{\delta,\varepsilon}=
 (\widetilde M\times\Pin^\delta(2))/\Z,\qquad \varepsilon\in\{+1,-1\},
\]
where the generator acts by
\begin{equation}\label{eq:principal}
 (p,g)\longmapsto(\tau p,\varepsilon e_w g).
\end{equation}
Their lifted generators square to $(p,g)\mapsto(\tau^2p,\delta g)$. Changing $\varepsilon$ is the nontrivial torsor action.
\end{theorem}
\begin{proof}
In the coordinate trivialization,
\[
 O(TM)=(\widetilde M\times O(2))/((p,h)\sim(\tau p,Ah)).
\]
Since $p_\delta(\varepsilon e_w)=A$, the map $[p,g]\mapsto[p,p_\delta(g)]$ is a well-defined equivariant double cover of this frame bundle. The action \eqref{eq:principal} commutes with the right Pin action. Its square is as asserted because $(\varepsilon e_w)^2=\delta$; this is allowed because the deck group is $\Z$, with no order-two relation on its generator.

Conversely, pull a Pin structure back to the contractible strip and lift the coordinate frame to a global Pin frame. The lift of $\tau$ must multiply this frame by an element over $A$. There are exactly two such elements, $e_w$ and $-e_w$, and continuity makes the choice constant. A change of lifted coordinate frame over the same orthogonal frame is a constant sign, which does not change this choice. This gives exactly the two structures. Twisting by the nontrivial real flat line, with character $n\mapsto(-1)^n$, changes $\varepsilon$ to $-\varepsilon$.
\end{proof}

\begin{remark}
Inequivalent Pin lifts should not be confused with inequivalent underlying complex vector bundles. Since $M$ retracts onto a circle, the latter can be trivial even when the Pin data or flat connection differs.
\end{remark}

\section{Covariant pinor representations and equivariant sections}
Use the Hermitian Pauli matrices
\[
 \sx=\begin{pmatrix}0&1\\1&0\end{pmatrix},\quad
 \sy=\begin{pmatrix}0&-i\\i&0\end{pmatrix},\quad
 \sz=\begin{pmatrix}1&0\\0&-1\end{pmatrix}.
\]
Set $\gamma(v)=v_x\sx+v_w\sy$ and $\Gamma=\sz$. We fix the Clifford representations
\begin{equation}\label{eq:rho}
 \rho_+(v)=-i\gamma(v)\Gamma,\qquad
 \rho_-(v)=\gamma(v)\Gamma.
\end{equation}
They satisfy the defining relations of $\mathrm{Cl}_+(2)$ and $\mathrm{Cl}_-(2)$, respectively. In particular,
\[
\begin{array}{c|cc}
 &e_x&e_w\\ \hline
\rho_+&-\sy&\sx\\
\rho_-&-i\sy&i\sx
\end{array}
\]
and $\rho_\delta(-1)=-I$. Both restrict to unitary Pin representations.

\begin{lemma}[Clifford covariance]\label{lem:covariance}
For every $g\in\Pin^\delta(2)$ and $v\in\R^2$,
\begin{equation}\label{eq:covariance}
 \rho_\delta(g)\gamma(v)\rho_\delta(g)^{-1}
 =\gamma(p_\delta(g)v).
\end{equation}
Thus $-i\gamma$ gives a skew-Hermitian Clifford multiplication on the associated bundle.
\end{lemma}
\begin{proof}
It suffices to check a unit vector $u$, since such vectors generate the Pin group. The matrix $\gamma(u)\Gamma$ anticommutes with $\gamma(u)$ and commutes with $\gamma(v)$ for $v\perp u$. Conjugation therefore changes the normal component of $v$ by a sign and fixes its perpendicular component. Multiplication of the conjugating matrix by $-i$ does not affect its conjugation action. This proves \eqref{eq:covariance}. Also $\gamma(v)^2=|v|^2I$, so $(-i\gamma(v))^2=-|v|^2I$.
\end{proof}

The distinction between the representation \eqref{eq:rho} and the Clifford multiplication $-i\gamma$ is deliberate. Direct use of the same Clifford generators for both roles, without checking the twisted adjoint action, can produce an operator taking values in the orientation-line twist of the intended bundle. Formula \eqref{eq:covariance} specifies the endomorphism-valued Dirac convention used here; compare \cite{CGT}.

Let $S_{\delta,\varepsilon}=P_{\delta,\varepsilon}\times_{\rho_\delta}\C^2$, and put
\begin{equation}\label{eq:c}
 c=\begin{cases}\varepsilon,&\delta=+1,\\i\varepsilon,&\delta=-1,\end{cases}
 \qquad R_c=c\sx.
\end{equation}
We also write $S_c$ for this bundle with its flat connection.

\begin{theorem}[Equivariant-section correspondence]\label{thm:sections}
Smooth sections of $S_c$ correspond bijectively to smooth functions on the strip satisfying
\begin{equation}\label{eq:equiv}
 \Psi(x+1,-w)=R_c\Psi(x,w).
\end{equation}
This correspondence extends to $L^2$ and to nonnegative geometric Sobolev orders. The $L^2$ norm is the integral over $F=[0,1]\times[-1,1]$. Furthermore,
\[
 R_c^2=c^2I=\delta I,\quad \spec(R_c)=\{c,-c\},\quad
 E_{\pm c}(R_c)=E_{\pm1}(\sx).
\]
\end{theorem}
\begin{proof}
The associated quotient is $(\widetilde M\times\C^2)/((p,v)\sim(\tau p,R_cv))$. A section gives a unique lift in this trivialization, and equality of its two descriptions at $p$ and $\tau p$ is exactly \eqref{eq:equiv}. Conversely, \eqref{eq:equiv} makes $[p,\Psi(p)]$ independent of the chosen lift of a base point. Smoothness is a local property of these descriptions. Unitarity of $R_c$ makes the norm independent of the fundamental domain. The same argument with measurable functions or weak derivatives gives the stated extensions. The matrix statements follow from $\sx^2=I$ and its two one-dimensional eigenspaces.
\end{proof}

In particular $\Psi(x+2,w)=\delta\Psi(x,w)$. On the orientation double cover the pulled-back fields are periodic over length two for $\Pin^+$ and antiperiodic over length two for $\Pin^-$. This is not yet the full one-step Fourier condition.

\section{Fourier decomposition of equivariant sections}\label{sec:fourier}
Let $\mathcal K=L^2([-1,1];\C^2)$ and define
\[
 (\mathcal P\phi)(w)=\phi(-w),\quad J=\sx\mathcal P,\quad
 \mathcal K_j=\ker(J-jI),\quad \Pi_j=\tfrac12(I+jJ).
\]
Both $\mathcal P$ and $J$ are unitary involutions. The one-step condition is
\begin{equation}\label{eq:one-step}
 \Psi(x+1,\cdot)=cJ\Psi(x,\cdot).
\end{equation}
Define $\theta_{c,j}$ by \eqref{eq:phaseintro} and $q_{c,j,k}=2\pi(k+\theta_{c,j})$.

\begin{theorem}[Equivariant Fourier decomposition]\label{thm:fourier}
Every $L^2$ section has the unique expansion
\begin{equation}\label{eq:Fourier}
 \Psi(x,w)=\sum_{j=\pm1}\sum_{k\in\Z}
 e^{iq_{c,j,k}x}\phi_{j,k}(w),\qquad \phi_{j,k}\in\mathcal K_j,
\end{equation}
with convergence in $L^2(F)$ and
\begin{equation}\label{eq:parseval}
 \norm{\Psi}_{L^2(M)}^2=\sum_{j,k}\norm{\phi_{j,k}}_{\mathcal K}^2.
\end{equation}
Conversely, every square-summable coefficient family of this form defines a section. The phase assignments are
\begin{center}
\begin{tabular}{ccccc}
\toprule
Pin type&$\varepsilon$&$R_c$&$\theta_{c,+1}$&$\theta_{c,-1}$\\
\midrule
$+$&$+1$&$\sx$&$0$&$1/2$\\
$+$&$-1$&$-\sx$&$1/2$&$0$\\
$-$&$+1$&$i\sx$&$1/4$&$3/4$\\
$-$&$-1$&$-i\sx$&$3/4$&$1/4$\\
\bottomrule
\end{tabular}
\end{center}
\end{theorem}
\begin{proof}
Apply $\Pi_j$ to \eqref{eq:one-step}. The result has scalar quasiperiod $cj$. Multiplication by $e^{-2\pi i\theta_{c,j}x}$ converts it into a periodic $\mathcal K_j$-valued function of period one. The ordinary Hilbert-valued Fourier theorem gives \eqref{eq:Fourier} and Parseval on each summand. The summands are orthogonal because $\mathcal K_+\perp\mathcal K_-$, even though the two scalar shifted exponential systems are not mutually orthogonal on $[0,1]$. Reversing the construction proves the converse. The table is obtained by taking the arguments of $cj$.
\end{proof}

If a scalar transverse function has reflection parity $p=\pm1$ and a constant fiber vector satisfies $\sx v=a v$, then their product belongs to $\mathcal K_{ap}$. Its phase is therefore $cap$. The eigenvalue of the fiber matrix alone determines the phase only after the transverse parity has been fixed.

The full momentum set is $\frac12\Z$ in the positive case and $\frac14+\frac12\Z$ in the negative case. The antiperiodic family $e^{2\pi i(k+1/2)x}$ occurs precisely on $\mathcal K_{-\varepsilon}$ for a positive Pin lift. It is a proper reducing sector, not the entire positive Pinor space.

\begin{corollary}[Additional flat holonomy]\label{cor:holonomy}
Tensoring with a flat unitary line of holonomy $e^{2\pi i\beta}$ around the core replaces $c$ by $e^{2\pi i\beta}c$ and adds $\beta$ to both phases modulo one. Arbitrary phases obtained this way include additional connection data and are not additional Pin structures.
\end{corollary}
\begin{proof}
The tensor-product deck representation multiplies $R_c$ by the scalar holonomy. Apply Theorem~\ref{thm:fourier} to the resulting one-step operator.
\end{proof}

\section{Geometric Sobolev spaces}\label{sec:sobolev}
For an integer $m\geq0$, define $H_c^m(M;S_c)$ using the flat connection and all weak covariant derivatives of order at most $m$. No physical boundary condition is imposed in this definition. Equivalently, use equivariant distributions on the strip whose derivatives have finite $L^2(F)$ norms. For $s=m+\alpha\geq0$, $0<\alpha<1$, use complex interpolation between adjacent integer orders. These are the usual nonnegative bundle Sobolev spaces on the compact manifold with boundary.

\begin{theorem}[Sobolev decomposition and embeddings]\label{thm:sobolev}
For each integer $m\geq0$, the Fourier transform in Theorem~\ref{thm:fourier} identifies $H_c^m$ with the coefficient families satisfying
\begin{equation}\label{eq:sobolevnorm}
 \sum_{j,k}\sum_{r=0}^{m}(1+q_{c,j,k}^2)^{m-r}
       \norm{\partial_w^r\phi_{j,k}}_{L^2(-1,1)}^2<\infty.
\end{equation}
The square root of this expression is equivalent to the geometric norm. All $H_c^s$, $s\geq0$, are complete, and smooth equivariant sections are dense. For $t\geq s\geq0$, $H_c^t\hookrightarrow H_c^s$ continuously; for $t>s$ the inclusion is compact. Translation in $x$ and $J$ act continuously and preserve the sector decomposition at every order.
\end{theorem}
\begin{proof}
For smooth sections, apply Parseval to every $\partial_x^a\partial_w^r$ with $a+r\leq m$. Differentiation in $w$ changes $\mathcal K_j$ to $\mathcal K_{(-1)^rj}$, preserving orthogonality of the two summands. For each fixed $r$, the longitudinal weight is $\sum_{a=0}^{m-r}q^{2a}$, uniformly equivalent to $(1+q^2)^{m-r}$. This proves \eqref{eq:sobolevnorm}; the distributional characterization gives the converse as well.

Here are analytic details concerning descent. On the length-two cylinder consider fields satisfying $\Psi(x+2,w)=c^2\Psi(x,w)$. They are sections of a scalar flat twist of a trivial rank-two bundle. In these spaces
\[
 (G\Psi)(x,w)=R_c^{-1}\Psi(x+1,-w)
\]
is an isometric involution. The bounded projection $(I+G)/2$ has range exactly the descended sections. Integer Sobolev spaces on this cylinder are complete by the weak-derivative definition. Smooth approximation near an interval endpoint can be obtained by a bounded extension, followed by convolution; a finite partition of unity gives global smooth approximation. Projecting these approximants by $(I+G)/2$ preserves convergence and imposes equivariance. This proves completeness and density downstairs. Local trivializations have constant unitary transition matrices, so the derivative norms just used are equivalent to the chart-and-partition definitions of the bundle norms.

For clarity, the compactness assertion uses both variables. On any bounded range of Sobolev orders, extend the transverse interval to a larger circle by a bounded interval extension and a cutoff. Untwist the length-two longitudinal boundary condition by multiplication by a fixed smooth phase. On the resulting torus the Sobolev norms are equivalent to weights $(1+\ell^2+n^2)^s$. The inclusion from order $t$ to $s<t$ is a norm limit of finite Fourier truncations, since the tail bound is $O(N^{s-t})$. Restriction and the equivariance projection prove compactness on $M$. The same extension, restriction and Fourier argument identifies the interpolation spaces with the usual fractional Sobolev spaces. Continuity of embeddings follows from the weights and interpolation. Translation commutes with all derivatives and with $G$. Reflection changes derivative signs, and multiplication by $\sx$ is unitary; hence $J$ preserves the norms and their interpolations.
\end{proof}

\begin{remark}[Longitudinal regularity is not full regularity]\label{rem:anisotropic}
The norm $\sum_{j,k}(1+q_{c,j,k}^2)^s\norm{\phi_{j,k}}_{L^2}^2$ measures only regularity in $x$. Its inclusion into $L^2$ is not compact: fix one longitudinal mode and choose an orthonormal sequence in its infinite-dimensional $\mathcal K_j$. Also, the product weight $(1+q^2)^s\norm{\phi}_{H^s_w}^2$ describes stronger mixed regularity, not the isotropic two-dimensional norm \eqref{eq:sobolevnorm}. Negative orders and boundary traces require a choice of dual or completion; we do not identify all such choices. An all-real operator scale is specified separately in Section~\ref{sec:scale}.
\end{remark}

\section{The Dirac expression and its physical boundary domain}\label{sec:domain}
The flat connection and the Clifford multiplication from Lemma~\ref{lem:covariance} give
\begin{equation}\label{eq:D}
 D=-i(\sx\partial_x+\sy\partial_w),\qquad
 D^2=-(\partial_x^2+\partial_w^2)I.
\end{equation}
The principal symbol is $\gamma(\xi)=\xi_x\sx+\xi_w\sy$, whose square is $|\xi|^2I$. Thus $D$ is elliptic. The identities
\[
 R_c\sx=\sx R_c,\qquad R_c\sy=-\sy R_c
\]
show, by differentiating \eqref{eq:equiv}, that $D$ preserves equivariance. The factor $-i$ in \eqref{eq:D} is part of the formally self-adjoint convention.

Let $t=\partial_x$ be the global unit longitudinal vector field and fix $b\in\{+1,-1\}$. The physical boundary condition is
\begin{equation}\label{eq:BC}
 (I-b\gamma(t))\Psi|_{\partial M}=0,
 \quad\text{equivalently}\quad \sx\Psi(x,\pm1)=b\Psi(x,\pm1).
\end{equation}
Both lifts of the boundary use the same eigenspace, and $R_c$ preserves it. This is one condition on the single boundary component. It is distinct from the identification \eqref{eq:equiv}.

\begin{definition}\label{def:realization}
The Hilbert space is $\mathcal H_c=L^2(M;S_c)$, with inner product conjugate-linear in the first entry. Define
\[
 \Dom(D_{c,b})=\{\Psi\in H_c^1:(I-b\sx)\Psi(x,\pm1)=0\},
 \qquad D_{c,b}\Psi=D\Psi.
\]
Boundary equalities are trace equalities. For $s\geq0$, write $H_{c,b}^{s+1}$ for the subspace of $H_c^{s+1}$ satisfying the same trace condition.
\end{definition}

\begin{proposition}[Boundary symmetry and ellipticity]\label{prop:boundary}
Condition \eqref{eq:BC} is a local self-adjoint elliptic boundary condition in the boundary-symbol sense.
\end{proposition}
\begin{proof}
Green's identity, with the stated inner-product convention, is
\[
 \ip{Du}{v}-\ip{u}{Dv}=i\int_{\partial M}u^*\gamma(\nu)v,
\]
where $\nu$ is the outward normal. On the two lifted edges $\gamma(\nu)=\pm\sy$. Since $\sy$ interchanges the $\sx$ eigenspaces, the line $E_b(\sx)$ is isotropic and equal to its annihilator for this form. The seam terms cancel by equivariance and $R_c^*\sx R_c=\sx$.

For the complementing condition, freeze a nonzero tangential frequency $\xi$ and use an inward half-line normal coordinate. The decaying solutions of the zero-energy principal ODE have boundary values in one of the two $\sz$ eigenspaces; changing the sign of $\xi$ or the edge interchanges them. Neither coordinate line $E_{\pm1}(\sz)$ intersects $E_b(\sx)$ nontrivially. Each has dimension one, so the boundary projection is an isomorphism on the decaying subspace. This is the Shapiro--Lopatinski condition. The actual adjoint domain, rather than only its boundary symbol, is proved next.
\end{proof}

\section{Closedness, Fredholmness and compact resolvent}\label{sec:fredholm}
We give a direct proof for \eqref{eq:BC}. It also specifies exactly which regularity is controlled by the graph norm.

\begin{lemma}[Transverse self-adjointness]\label{lem:transverse}
On $\mathcal K$, let
\[
 T_b=-i\sy\frac{d}{dw},\qquad
 \Dom(T_b)=\{\phi\in H^1([-1,1];\C^2):\sx\phi(\pm1)=b\phi(\pm1)\}.
\]
Then $T_b$ is self-adjoint. The involution $J$ preserves its domain and commutes with it. For every $q\in\R$, the operator
\[
 D_q=q\sx+T_b
\]
with the same domain is self-adjoint, reduces $\mathcal K_\pm$, and satisfies
\begin{equation}\label{eq:energy}
 \norm{D_q\phi}^2=q^2\norm{\phi}^2+\norm{\phi'}^2.
\end{equation}
\end{lemma}
\begin{proof}
The boundary form of $T_b$ vanishes on the specified endpoint lines. If $v\in\Dom(T_b^*)$, testing against compactly supported functions shows distributionally that $v'\in L^2$, hence $v\in H^1$. The endpoint values of test functions in $\Dom(T_b)$ vary independently in $E_b(\sx)$. Green's identity therefore forces $v(\pm1)$ into the annihilator of this line for $\sy$, which is the same line. This proves equality of the domains of $T_b$ and its adjoint.

Reflection reverses the derivative, and $\sx\sy=-\sy\sx$; thus $J$ commutes with $T_b$. It also preserves the endpoint condition. The bounded Hermitian operator $q\sx$ preserves this domain and commutes with $J$. The adjoint of its sum with $T_b$ has the domain of $T_b^*$, proving self-adjointness of $D_q$. The projections $(I\pm J)/2$ commute with the operator and preserve the domain, so their restrictions are self-adjoint. Expanding the squared norm gives the cross term
\[
 2q\operatorname{Re}\int_{-1}^1\phi^*\sz\phi'
 =q[\phi^*\sz\phi]_{-1}^{1}=0,
\]
where the last equality follows because $\sz$ interchanges the $\sx$ eigenspaces. This proves \eqref{eq:energy}.
\end{proof}

\begin{theorem}[Closed self-adjoint realization]\label{thm:SA}
The operator $D_{c,b}$ is densely defined and self-adjoint. Its domain is exactly Definition~\ref{def:realization}, and
\begin{equation}\label{eq:graph}
 \norm{\Psi}^2+\norm{D_{c,b}\Psi}^2
 =\norm{\Psi}^2+\norm{\partial_x\Psi}^2+\norm{\partial_w\Psi}^2.
\end{equation}
Smooth equivariant sections satisfying \eqref{eq:BC} form a core. In particular, their operator is closable and its closure is $D_{c,b}$.
\end{theorem}
\begin{proof}
Use Theorem~\ref{thm:fourier}. In its $(j,k)$ summand the differential expression is
\[
 D_{q_{c,j,k}}|_{\mathcal K_j}.
\]
By Lemma~\ref{lem:transverse} each of these is self-adjoint on the intersection of $\Dom(T_b)$ with $\mathcal K_j$. Their Hilbert direct sum is self-adjoint on the domain where both the coefficients and their images are square summable. To see this without a domain assumption, the resolvent at $i$ of each summand has norm at most one, and their direct sum is an everywhere-defined inverse; the same is true at $-i$. A densely defined symmetric direct sum with both these ranges equal to the Hilbert space is self-adjoint.

Summing \eqref{eq:energy} identifies its domain with the coefficient condition for $H_c^1$ together with \eqref{eq:BC}, and proves \eqref{eq:graph}. This also proves closedness directly. For the core assertion first truncate the longitudinal expansion in graph norm. In each of the finitely many transverse domains, smooth $H^1$ approximation can be followed by linear endpoint corrections to put both traces in $E_b(\sx)$. The corrections tend to zero in $H^1$ because the trace map is continuous; finally apply $(I+jJ)/2$. The resulting smooth functions converge in the same graph norm. Such a domain is dense in $L^2$, for example by first approximating with functions supported away from the physical boundary and then imposing equivariance.
\end{proof}

\begin{theorem}[Compact resolvent and Fredholm property]\label{thm:compact}
The resolvent of $D_{c,b}$ is compact. Its spectrum is real and discrete, each eigenvalue has finite multiplicity, and there is an orthonormal basis of eigenvectors. Its kernel is finite dimensional, its range is closed, and
\[
 \ran D_{c,b}=(\ker D_{c,b})^\perp,\qquad \ind D_{c,b}=0.
\]
Here the index is the ordinary complex Fredholm index of the ungraded operator.
\end{theorem}
\begin{proof}
Self-adjointness gives $(D_{c,b}-i)^{-1}$ on all of $\mathcal H_c$. The identity
\[
 \norm{(D_{c,b}-i)\Psi}^2=\norm{D_{c,b}\Psi}^2+\norm{\Psi}^2
\]
and \eqref{eq:graph} make this inverse bounded into $H_c^1$. The inclusion $H_c^1\hookrightarrow L^2$ is compact by Theorem~\ref{thm:sobolev}. Thus the resolvent is compact. The compact normal spectral theorem applied to the resolvent gives the stated discrete spectral decomposition. Only finitely many eigenvalues can lie in a bounded interval, so the nonzero spectrum has a positive distance from zero. The inverse on the orthogonal complement of the kernel is bounded, which proves closed range. Self-adjointness identifies the cokernel with the kernel and gives index zero.
\end{proof}

\begin{theorem}[Geometric Sobolev Fredholm mappings]\label{thm:HsFredholm}
For every real $s\geq0$,
\begin{equation}\label{eq:HsFredholm}
 D:H_{c,b}^{s+1}\longrightarrow H_c^s
\end{equation}
 is bounded and Fredholm of index zero. As an unbounded operator on $H_c^s$ with this domain, it is closed. The kernel consists of smooth sections and is independent of $s$.
\end{theorem}
\begin{proof}
We first prove, for integers $m\geq0$, that $u\in\Dom(D_{c,b})$ and $f=Du\in H_c^m$ imply
\begin{equation}\label{eq:regularity}
 \norm{u}_{H_c^{m+1}}\leq C_m(\norm{f}_{H_c^m}+\norm{u}_{L^2}).
\end{equation}
For a finite longitudinal truncation apply \eqref{eq:energy} to $q^a u_{j,k}$ for $0\leq a\leq m$. This controls all derivatives with at most one $w$ derivative and total order at most $m+1$. Constants with $q=0$ are controlled by the $L^2$ term. For the remaining derivatives use the distributional equation
\begin{equation}\label{eq:normalode}
 \partial_wu_{j,k}=i\sy f_{j,k}-q\sz u_{j,k}.
\end{equation}
Differentiating this identity controls a derivative with $r\geq2$ normal derivatives by derivatives of $f$ of total order at most $m$ and derivatives of $u$ with one fewer normal derivative and the same total order. Induction on $r$ proves \eqref{eq:regularity}, uniformly in the truncation. Passing to the limit proves both the estimate and the asserted gain of regularity.

For $D+i$, the $L^2$ inverse exists by self-adjointness. If $g\in H_c^m$ and $(D+i)u=g$, induction in \eqref{eq:regularity} applied to $Du=g-iu$ gives $u\in H_c^{m+1}$ and a bounded inverse from $H_c^m$ to $H_c^{m+1}$. Interpolation of this single consistent inverse at adjacent integer orders gives the same statement for every real $s\geq0$. The output satisfies the boundary condition because it is the original $L^2$ inverse and its trace is defined. Conversely $D+i$ maps $H_{c,b}^{s+1}$ to $H_c^s$ and is injective. It is therefore an isomorphism between these spaces.

The inclusion $H_{c,b}^{s+1}\hookrightarrow H_c^s$ is compact. The map in \eqref{eq:HsFredholm} is a compact perturbation of the isomorphism $D+i$, so it is Fredholm of index zero. The corresponding estimate with $\norm{u}_{H_c^s}$ gives equivalence of its graph norm and the $H_c^{s+1}$ norm and proves closedness. Iterating \eqref{eq:regularity} with $f=0$ makes kernel elements smooth.
\end{proof}

\section{Spectrum, symmetry and boundary dependence}\label{sec:spectrum}
\subsection{Transverse quantization}
For $b=\pm1$, let $v_b=2^{-1/2}(1,b)^T$. On the interval of length two set
\begin{equation}\label{eq:transversemodes}
 \mu_n=\frac{n\pi}{2},\qquad
 f_n(w)=\frac{1}{\sqrt2}\exp\big(i\sy\mu_n(w+1)\big)v_b,
 \qquad n\in\Z.
\end{equation}

\begin{lemma}\label{lem:transversespectrum}
The functions $f_n$ are an orthonormal eigenbasis of $T_b$, with eigenvalues $\mu_n$. Moreover,
\begin{equation}\label{eq:transverseidentities}
 Jf_n=b(-1)^nf_n,\qquad \sx f_n=b f_{-n}.
\end{equation}
The zero transverse eigenvalue is simple.
\end{lemma}
\begin{proof}
The equation $T_b\phi=\mu\phi$ gives $\phi(w)=\exp(i\sy\mu(w+1))\phi(-1)$. The condition at $-1$ makes $\phi(-1)$ a multiple of $v_b$. The condition at $1$ is $\sin(2\mu)=0$, because $\sy v_b$ lies in the opposite $\sx$ eigenspace. Self-adjointness makes $\mu$ real, giving exactly \eqref{eq:transversemodes}. The normalization is correct since the exponential is unitary and the interval length is two. Distinct eigenvalues are orthogonal. The graph norm of $T_b$ is the interval $H^1$ norm, so its resolvent is compact. Its spectral theorem and the solved eigenvalue equation imply completeness. Anticommutation of $\sx$ and $\sy$, together with $\exp(2i\sy\mu_n)=(-1)^nI$, proves \eqref{eq:transverseidentities}.
\end{proof}

\subsection{The full spectrum}
Put
\begin{equation}\label{eq:thetan}
 e^{2\pi i\theta_n}=cb(-1)^n,\qquad 0\leq\theta_n<1,\qquad
 q_{n,k}=2\pi(k+\theta_n).
\end{equation}

\begin{theorem}[Coupled spectral formula]\label{thm:spectrum}
The spectrum of $D_{c,b}$, including multiplicities, is the multiset union of
\begin{align}
 \lambda_{0,k}&=bq_{0,k},&&k\in\Z,\label{eq:zerochannel}\\
 \lambda_{n,k}^{\pm}&=\pm\sqrt{q_{n,k}^2+(n\pi/2)^2},
 &&n\geq1,\ k\in\Z.\label{eq:pairedchannels}
\end{align}
Each displayed branch contributes one eigenvector for each of its labels; coincident values contribute the sum of these multiplicities. The transverse zero mode occurs only once. In particular,
\begin{equation}\label{eq:kernel}
 \dim\ker D_{c,b}=\begin{cases}1,&cb=1,\\0,&cb\ne1.\end{cases}
\end{equation}
\end{theorem}
\begin{proof}
By \eqref{eq:transverseidentities}, the span of $f_n,f_{-n}$ for $n\geq1$ lies in $\mathcal K_{b(-1)^n}$. Equivariance therefore imposes exactly the phase \eqref{eq:thetan}. On this two-dimensional span the $k$th longitudinal block is
\begin{equation}\label{eq:block}
 \begin{pmatrix}\mu_n&bq_{n,k}\\bq_{n,k}&-\mu_n\end{pmatrix}.
\end{equation}
Its eigenvalues are \eqref{eq:pairedchannels}. The space spanned by $f_0$ is one dimensional, and there $T_b=0$ and $\sx=b$, giving \eqref{eq:zerochannel}. Completeness follows by first using the longitudinal Fourier decomposition and then the transverse eigenbasis in each $J$ eigenspace. No branch with $n\geq1$ can vanish. In the zero branch, $k+\theta_0=0$ is possible exactly when $\theta_0=0$, equivalently $cb=1$, and then $k=0$ is unique.
\end{proof}

The factors $2\pi$ and $\pi/2$ arise, respectively, from the core length one and transverse interval length two. The phase assignment alternates with $n$; replacing all $\theta_n$ by one half discards part of the spectrum. Also, eliminating the longitudinal zero frequency does not in general determine the kernel for a different physical boundary condition.

\subsection{Domain-preserving and sector-exchanging symmetries}
Let $K$ denote entrywise complex conjugation in the chosen trivialization. Define the antilinear map $C=\sx K$. Its action on equivariance is obtained by conjugating the scalar $c$.

\begin{proposition}[Symmetry with domains]\label{prop:symmetry}
The following maps have the indicated domains and intertwining relations:
\begin{center}
\begin{tabular}{lll}
\toprule
Map&Domain mapping&Relation\\
\midrule
$C=\sx K$&$\Dom D_{c,b}\longrightarrow\Dom D_{\bar c,b}$
 &$D_{\bar c,b}C=-CD_{c,b}$\\
$\Gamma=\sz$&$\Dom D_{c,b}\longrightarrow\Dom D_{-c,-b}$
 &$D_{-c,-b}\Gamma=-\Gamma D_{c,b}$\\
\bottomrule
\end{tabular}
\end{center}
For $c=\pm1$, $C$ preserves the realization and gives spectral symmetry within it. For $c=\pm i$, $C$ exchanges the two Pin lifts with the same $b$. The constant chirality matrix never preserves both the lift and the boundary sign in this family.
\end{proposition}
\begin{proof}
The identities $C\sx=\sx C$ and $C\sy=\sy C$, together with $Ci=-iC$, give $CDC^{-1}=-D$. Also
\[
 (C\Psi)(\tau p)=\bar c\sx(C\Psi)(p),\qquad
 \sx(C\Psi)|_{\partial M}=b(C\Psi)|_{\partial M}.
\]
Thus the first domain mapping holds. For $\Gamma$, anticommutation with both $\sx$ and $\sy$ gives $\Gamma D=-D\Gamma$, equivariance with $-c$, and boundary eigenvalue $-b$. All maps preserve $H^1$. Their inverses have the same properties, giving bijective mappings of the stated domains and eigenspaces.
\end{proof}

\subsection{Eta values and the ungraded index}
For $\operatorname{Re}s>2$, define the eta function by the absolutely convergent sum over nonzero eigenvalues. The paired branches cancel. If $0<\theta_0<1$, the remaining branch gives
\begin{equation}\label{eq:eta}
 \eta_{c,b}(s)=b(2\pi)^{-s}
 \big[\zeta(s,\theta_0)-\zeta(s,1-\theta_0)\big].
\end{equation}
If $\theta_0=0$, the nonzero eigenvalues of this branch pair and its eta function is identically zero. The identity $\zeta(0,a)=1/2-a$ and meromorphic continuation of the Hurwitz zeta function \cite[Section 25.11]{DLMF} therefore give
\begin{equation}\label{eq:eta0}
 \eta_{c,b}(0)=\begin{cases}
 0,&\theta_0=0,\\
 b(1-2\theta_0),&0<\theta_0<1.
 \end{cases}
\end{equation}
This derivation also proves regularity at zero for these realizations.

\begin{corollary}\label{cor:pineta}
For either boundary sign $b$, the four pure Pin lifts satisfy
\begin{center}
\begin{tabular}{cccc}
\toprule
Type&$c$&$\dim\ker D_{c,b}$&$\eta_{c,b}(0)$\\
\midrule
$\Pin^+$&$+1$&$1$ if $b=+1$, otherwise $0$&$0$\\
$\Pin^+$&$-1$&$1$ if $b=-1$, otherwise $0$&$0$\\
$\Pin^-$&$+i$&$0$&$+1/2$\\
$\Pin^-$&$-i$&$0$&$-1/2$\\
\bottomrule
\end{tabular}
\end{center}
Each individual negative Pin realization has an asymmetric spectrum. Their direct sum at fixed $b$ has symmetric spectrum and zero eta function. All four have ordinary complex Fredholm index zero.
\end{corollary}
\begin{proof}
Insert the phases from \eqref{eq:thetan} into \eqref{eq:kernel} and \eqref{eq:eta0}. The nonzero eta values exclude spectral symmetry for an individual negative lift. Alternatively, its zero branch has different positive and negative offsets, whereas every nonzero transverse branch is paired. Proposition~\ref{prop:symmetry} pairs the two negative lifts. The index statement is Theorem~\ref{thm:compact}.
\end{proof}

The sign labels in this table depend on the chosen generator, representation and sign of $D$. The eta values are properties of these boundary realizations, not a boundary-independent invariant of an unspecified Pin surface. We have not defined a chiral index or a real mod-two index. In particular, a symmetric spectrum would not imply a trivial kernel.

\subsection{The translation-invariant local boundary family}\label{sec:family}
The domain dependence can be made explicit beyond \eqref{eq:BC}. Let $t$ be an angle modulo $2\pi$, and put
\begin{equation}\label{eq:Bt}
 B_+(t)=\cos t\,\sx+\sin t\,\sz,\qquad
 B_-(t)=\sx B_+(t)\sx=\cos t\,\sx-\sin t\,\sz.
\end{equation}
Impose $B_+(t)\Psi(x,1)=\Psi(x,1)$ and $B_-(t)\Psi(x,-1)=\Psi(x,-1)$. Write $D_{c,t}$ for the expression \eqref{eq:D} on the resulting $H_c^1$ domain. The special cases $t=0,\pi$ are $D_{c,+1}$ and $D_{c,-1}$.

\begin{proposition}[Boundary classification in the invariant family]\label{prop:classification}
Every translation-invariant rank-one local boundary condition which is self-annihilating for Green's form and compatible with the deck action is of the form \eqref{eq:Bt}. It is elliptic exactly when $\cos t\ne0$. In that case $D_{c,t}$ is self-adjoint with compact resolvent and ordinary index zero.
\end{proposition}
\begin{proof}
An orthogonal rank-one projection in $\C^2$ has the form $(I+\mathbf n\cdot\boldsymbol\sigma)/2$, $|\mathbf n|=1$. Isotropy for $\sy$ is equivalent to $n_2=0$. Thus the line at the upper boundary is the positive eigenspace of $B_+(t)$ for some $t$. Deck compatibility conjugates this projection by $R_c$, giving $B_-(t)$. Conversely these lines are self-annihilating and compatible.

The decaying boundary lines of the principal ODE are the $\sz$ eigenspaces, as in Proposition~\ref{prop:boundary}. An allowed line equals one of them exactly when $\cos t=0$. Thus the complementing condition holds exactly as stated. For $\cos t\ne0$, the boundary lines form a smooth rank-one subbundle over the compact boundary, are self-annihilating for the normal symbol, and satisfy the local complementing condition at every nonzero tangential covector. The formally self-adjoint Dirac expression acts on a smooth Hermitian Clifford bundle over a compact smooth manifold. These are the hypotheses of the local ellipticity, adjoint-domain and regularity results in \cite[Theorems 4.15, 4.6, 4.9 and 4.11]{BBGuide}, in its specified version. They identify the self-adjoint domain with this $H^1$ boundary domain and give the graph estimate into $H^1$. Compactness of $H^1\hookrightarrow L^2$ then proves compact resolvent; closed range and index zero follow as in Theorem~\ref{thm:compact}. The general local classification is classical; compare \cite{GUV}.
\end{proof}

\begin{theorem}[Secular equation and kernel criterion]\label{thm:secular}
Suppose $\cos t\ne0$ and set
\[
 r_t=\frac{1-\sin t}{\cos t}\ne0.
\]
In the sector $j=\pm1$ and longitudinal frequency $q=2\pi(k+\theta_{c,j})$, the eigenvalues are exactly the real zeros of
\begin{equation}\label{eq:secular}
 F_{q,j,t}(\lambda)=(j-r_t)\cos\mu
 +\big[q(r_t+j)-\lambda(1+r_tj)\big]\frac{\sin\mu}{\mu},
 \qquad \mu^2=\lambda^2-q^2.
\end{equation}
The functions of $\mu$ in this formula are their entire extensions in $\mu^2$, including at $\mu=0$. Each sector-frequency eigenspace has dimension one. The full kernel has dimension at most one and is nonzero exactly when, for $j=\operatorname{sgn}(\cos t)$, there is $k\in\Z$ such that
\begin{equation}\label{eq:kercriterion}
 2\pi(k+\theta_{c,j})=\frac12\log|r_t|.
\end{equation}
\end{theorem}
\begin{proof}
The $J=j$ condition reduces the transverse problem to $[0,1]$, with $\sx\phi(0)=j\phi(0)$. Any nonzero solution starts at a multiple of $(1,j)^T$; zero initial data give the zero solution. The equation $D_q\phi=\lambda\phi$ is
\[
 \phi'=(i\lambda\sy-q\sz)\phi,
 \qquad (i\lambda\sy-q\sz)^2=(q^2-\lambda^2)I.
\]
Hence its fundamental matrix at $w=1$ is
\[
 \cos\mu\,I+\frac{\sin\mu}{\mu}(i\lambda\sy-q\sz).
\]
The upper boundary line has component ratio $\phi_2(1)/\phi_1(1)=r_t$, so applying the row $(-r_t,1)$ gives exactly \eqref{eq:secular}. Reflection reconstructs the other half-interval, and the initial $\sx$ condition and the ODE ensure matching at zero. Thus the scalar equation is both necessary and sufficient, including for imaginary $\mu$. It has a one-dimensional solution space at a zero, by uniqueness for the initial-value problem. Self-adjointness already proved in Proposition~\ref{prop:classification} ensures that the spectrum is real and has no algebraic multiplicity beyond its eigenspace multiplicity within a block.

At $\lambda=0$, the solution is proportional to $(e^{-qw},j e^{qw})^T$. The endpoint condition is $j e^{2q}=r_t$. Thus $j=\operatorname{sgn}(r_t)=\operatorname{sgn}(\cos t)$ and $q=\tfrac12\log|r_t|$. This gives \eqref{eq:kercriterion}. For fixed $t$ there is at most one such $j$ and one such $k$.
\end{proof}

This criterion illustrates the separate roles of lift and boundary: the lift fixes the discrete set of $q$ values, while the boundary line selects the real value to be matched. Formula \eqref{eq:pairedchannels} is the explicit special case $t=0,\pi$, not a universal formula for this family.

\section{Relation to abstract shifted Hilbert scales}\label{sec:scale}
For a prescribed $\theta$, the scalar scale generated by the modes $e^{2\pi i(k+\theta)x}$ is obtained by completing their finite span with coefficient norm
\[
 \sum_{k\in\Z}(1+|k+\theta|^2)^s|a_k|^2.
\]
For negative $s$ this completion need not consist of $L^2$ functions. For $s=0$ every such shifted system is an orthonormal basis of $L^2([0,1])$; the distinction between shifts lies in the global quasiperiodic interpretation and the domains of longitudinal operators. At order one the scalar domain is $u(1)=e^{2\pi i\theta}u(0)$.

The present construction derives $\theta$ before this abstract step. Each $J$ sector is a Hilbert-valued shifted system. Fixing a transverse profile gives a scalar longitudinal subspace, but it need not be invariant under $D$. The invariant zero-transverse subspace of \eqref{eq:zerochannel} is a particularly direct connection: there
\[
 D_{c,b}=b\left(-i\frac{d}{dx}\right),\qquad
 \spec\left(\frac{D_{c,b}}{2\pi b}\right)=\{k+\theta_0:k\in\Z\}.
\]
For $c=-1$ and $b=+1$, the phase is $\theta_0=1/2$, giving precisely the diagonal model studied abstractly in \cite{CK}. The full two-dimensional operator also contains all nonzero transverse branches, which cannot be replaced by this scalar model.

One may also construct an all-real scale from the actual boundary realization. If $\{u_\ell\}$ is its eigenbasis and $\lambda_\ell$ its eigenvalues, complete finite spectral sums in
\[
 \norm{u}_{\mathscr H_{c,b}^s}^2
 =\sum_\ell(1+\lambda_\ell^2)^s|\ip{u_\ell}{u}|^2.
\]
Then $D:\mathscr H_{c,b}^{s+1}\to\mathscr H_{c,b}^s$ is Fredholm for every real $s$: the zero-eigenspace is finite dimensional, and division by $\lambda_\ell$ is uniformly bounded between adjacent levels on its complement. At level one this is the graph domain, equivalent to $H_{c,b}^1$. Higher levels encode iterated boundary compatibility, and should not all be identified with an unconstrained geometric Sobolev space or with a space having only the single trace condition. This operator scale is separate from the nonnegative geometric mapping theorem, Theorem~\ref{thm:HsFredholm}.

\section{Conclusions and further questions}
For the flat M\"obius band, a Pin lift determines a unitary fiber action, but the longitudinal Fourier phase is the eigenvalue of its product with transverse reflection. Keeping that product explicit yields both sectors of each Pin type. A physical boundary condition is an additional datum. For the longitudinal Clifford boundary line, the resulting realization admits a direct graph estimate, compact resolvent, and a complete spectrum with a transverse-parity constraint. Its unpaired branch shows why one must verify domain preservation before invoking chirality or spectral symmetry.

Several questions remain beyond the model computed here. An $x$-dependent metric or boundary line generally couples the longitudinal modes, although the equivariant-section description persists. A useful extension would quantify how this coupling changes low eigenvalues and the gap under controlled geometric perturbations. For more general orientation-reversing quotients, one must replace the two reflection eigenspaces by the appropriate representation-theoretic decomposition. Overcomplete systems would require a separate analysis of coefficient redundancy; it is not supplied by an orthonormal Fourier decomposition. Finally, comparison of different elliptic boundary realizations requires their actual domains and cannot be inferred from the Pin lift alone.

\end{document}